\documentclass{article}
\usepackage{amscd}
\usepackage[all]{xy}
\usepackage{CJK}
\usepackage{CJKnumb}
\usepackage{graphicx}
\usepackage{amsfonts}
\usepackage{amsmath}
\usepackage{amssymb}
\usepackage{fancyhdr}
\usepackage{indentfirst}
\usepackage{titlesec}
\usepackage{mathrsfs}
\usepackage{bbm}
\usepackage{dsfont}
\usepackage{amsthm} 
\usepackage{mathrsfs}

\newcommand{\cP}{\mathcal{P}}
\newcommand{\cM}{\mathcal{M}}

\newcommand{\cR}{\mathcal{R}}

\newcommand{\cN}{\mathcal{N}}
\newcommand{\cA}{\mathcal{A}}

\newcommand{\cC}{\mathcal{C}}
\newcommand{\cG}{\mathcal{G}}
\newcommand{\cB}{\mathcal{B}}

\theoremstyle{definition}  \newtheorem{Def}{Definition}
\theoremstyle{plain}  \newtheorem{Thm}{Theorem} \theoremstyle{plain}

 \theoremstyle{remark} \newtheorem{Rek}{Remark}
 \theoremstyle{plain}\newtheorem{Lem}{Lemma}
 \theoremstyle{plain}

\theoremstyle{remark}

\begin{document}

\title{\textbf{Higher Dimensional Homology Algebra II:Projectivity
}}
\author{Fang Huang, Shao-Han Chen, Wei Chen, Zhu-Jun Zheng\thanks{Supported in part by NSFC with grant Number
10971071
 and Provincial Foundation of Innovative
Scholars of Henan.} }
\date{}
 \maketitle

\begin{center}
\begin{minipage}{5in}
{\bf  Abstract}:
In this paper, we will prove that the 2-category (2-SGp) of
symmetric 2-groups and 2-category ($\cR$-2-Mod) of $\cR$-2-modules(\cite{5})
have enough projective objects, respectively.
\\
\end{minipage}
\end{center}
\maketitle \hspace{1cm}

\section{Introduction}
The 2-category ($\cR$-2-Mod) of $\cR$-2-modules should be important like the category (R-Mod) in classical homology
algebra (we call it 1-dimensional homology algebra).
The property of projective enough of
the category
(R-Mod) is a stone for constructing derived functor and derived category\cite{3,7,10}. We believe that the property of projective enough of 2-category ($\cR$-2-Mod) play the same role in higher dimensional homology algebra as the category (R-Mod) in 1-dimensional homology algebra.

In \cite{1}, D. Bourn and E.M. Vitale gave the definition of
projective objects in the 2-category (2-SGp) of symmetric
categorical groups(we call them symmetric 2-groups) and said that
"another problem concerns projective objects (in the sense of
Definition 11.1) in the 2-category of symmetric categorical groups.
The notion of projectivity is crucial in the classical theory, but,
unfortunately, we do not know if the 2-category of symmetric
categorical groups has enough projective objects. (It would be
interesting to solve this problem in order to appreciate the strong
specialization done in Sections 14 and 15, where we consider only
$\mathcal{F}$-extensions.)"

The main aim of this paper is try to prove the conjecture of D. Bourn and
E. M. Vitale.  In fact, we prove that the 2-categories (2-SGp) and
($\cR$-2-Mod) have enough projective objects from the well-known
result that the abelian category (R-Mod) has enough projective
objects.

The present paper is organized as follows.
In section 2, we will recall some basic facts on symmetric 2-groups
and their extensions, which are appeared in \cite{1,2,11,6}, and
give the definition of projective objects in
($\cR$-2-Mod)(\cite{5}). In the next two sections, we will proof
(2-SGp) and ($\cR$-2-Mod) have enough projective objects.

This is the second paper of the series works on higher dimensional homology algebra. The first paper is "2-Modules and the Representation of
2-Rings\cite{4}". In the coming papers, we shall give the definition of injective object in the 2-category ($\cR$-2-Mod), prove that this 2-category has enough injective objects and develop the (co)homology theory of it.

\section{Preliminary}
In this section, we will give the basic definitions and results
cited from \cite{1,2,11,6}.

\begin{Def}\cite{1,11,6}
For a sequence $(\Gamma,\varphi,\Sigma)$ in (2-SGp) as in the
following diagram:
\begin{center}
\scalebox{0.9}[0.85]{\includegraphics{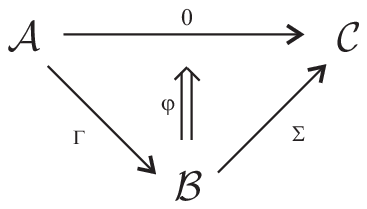}}
\end{center}
By the universal properties of kernel and cokernel(\cite{2,11,6}),
there are homomorphisms $\Gamma_0,\Sigma_0$ as in the following
diagram:
\begin{center}
\scalebox{0.9}[0.85]{\includegraphics{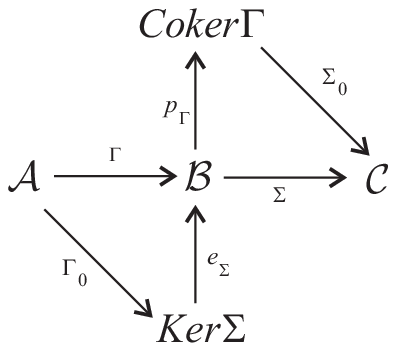}}
\end{center}
The sequence $(\Gamma,\varphi,\Sigma)$ is 2-exact if it satisfies
one of the following equivalent conditions:

1) $\Gamma_0:\cA\rightarrow Ker\Sigma$ is full and essentially
surjective;

2) $\Sigma_0:Coker\Gamma\rightarrow \cC$ is full and faithful.
\end{Def}

\begin{Rek} There are four equivalent conditions in above
definition from Proposition 6.2 in \cite{6}.
\end{Rek}
\begin{Def}\cite{1}
Let $\cA,\cC$ be in (2-SGp). An extension of $\cA$ by $\cC$ is a
diagram $(\Gamma,\varphi,\Sigma)$ in (2-SGp)
\begin{center}
\scalebox{0.9}[0.85]{\includegraphics{p1.eps}}
\end{center}
which satisfies the following equivalent conditions:

1) The triple $(\Gamma,\varphi,\Sigma)$ is 2-exact, $\Gamma$ is
faithful and $\Sigma$ is essentially surjective;

2) $\Gamma_0$ is an equivalence and $\Sigma$ is essentially
surjective;

3) $\Gamma$ is faithful and $\Sigma_0$ is an equivalence.
\end{Def}
Next, we will only consider one special case in the definition of
projective objects in (2-SGp) given by D. Bourn and E.M. Vitale in
\cite{1}.
\begin{Def}\cite{1}
Let $\cP$ be a symmetric 2-groups. $\cP$ is called projective if,
for each 1-morphism $G:\cP\rightarrow \cB$, and each essentially
surjective functor $F:\cA\rightarrow\cB$ in (2-SGp), there exist
$G^{'}:\cP\rightarrow\cA$, and $g:F\circ G^{'}\Rightarrow G$ in
(2-SGp).
\end{Def}
Similar as the methods in (2-SGp), we have
\begin{Def}
An object $\cP$ in ($\cR$-2-Mod)(\cite{5}) is called a projective
object, if for any $\cR$-homomorphism $G:\cP\rightarrow\cC$, and any
essentially surjective $\cR$-homomorphism $F:\cB\rightarrow\cC$,
there exist an $\cR$-homomorphism $G^{'}:\cP\rightarrow\cB$, and
2-morphism $h:F\circ G^{'}\Rightarrow G$ in ($\cR$-2-Mod).
\end{Def}

\section{Main Results I}

In this section, we will show that (2-SGp) has enough projective
objects from the basic results of 1-dimensional homological
algebraic theory.

\noindent\textbf{Notation}\cite{1,2,6}. For an abelian group
$G$, we write $G_{dis}$ for the symmetric 2-group with objects which are
the elements of $G$, morphism of $a\rightarrow b$ is only the
identity when $a=b$, the monoidal structure is induced from the
group structure of $G$. Moreover, for  a symmetric 2-group
$\cG$, we write $\pi_{0}(\cG)$ for the abelian group with the
elements which are objects of $\cG$ up to isomorphism(denote by $[b]$, for
$b\in obj(\cG)$), equipped with monoidal structure $+$ of $\cB$ as
the operation and with the unit object $0$ as the unit element.

\begin{Lem}
Given a surjective group homomorphism $f:B\rightarrow C$ of abelian
groups $B$ and $C$. There is an essentially surjective morphism
$F:B_{dis}\rightarrow C_{dis}$ of symmetric 2-groups.
\end{Lem}
\begin{proof}

There is a functor
\begin{align*}
&F:B_{dis}\longrightarrow C_{dis}\\
&\hspace{1.1cm}b\mapsto f(b),\\
&\hspace{0.3cm}b\xrightarrow[]{id}b\mapsto
f(b)\xrightarrow[]{id}f(b)
\end{align*}
Also,
$F(b_{1}+b_{2})=f(b_{1}+b_{2})=f(b_{1})+f(b_{2})=F(b_{1})+F(b_{2}).$
Then $F$ is homomorphism of symmetric 2-groups.

Then for any $c\in obj(C_{dis})$, i.e $c\in C$. From the surjective
group homomorphism $f$, there exists $b\in B$, such that $f(b)=c$.
Then there exists an object $b$ in $B_{dis}$, and identity morphism
$F(b)=c$. So, $F$ is essentially surjective.
\end{proof}
\begin{Lem}
Given an essentially surjective homomorphism $F:\cB\rightarrow\cC$
of symmetric 2-groups. There is a surjective group homomorphism
$F_{0}:\pi_{0}(\cB)\rightarrow\pi_{0}(\cC)$.
\end{Lem}
\begin{proof}

There is a group homomorphism
\begin{align*}
&F_{0}:\pi_{0}(\cB)\longrightarrow \pi_{0}(\cC)\\
&\hspace{1.2cm}[b]\mapsto [F(b)]
\end{align*}
which is well-defined, since if $b$ and $b^{'}$ are in same
equivalent class, i.e. there is an isomorphism $\alpha:b\rightarrow
b^{'}$ in $\cB$, and for $F$ is a functor, so there is an
isomorphism $F(\alpha):F(b)\rightarrow F(b^{'})$, then $F(b)$ and
$F(b^{'})$ are the same element in $\pi_{0}(\cC)$. Moreover, for any
$[b_1],[b_2]\in \pi_0(\cB)$,
$$
F_{0}([b_1]+[b_2])=F_{0}([b_1+b_2])=[F(b_{1}+b_{2})],
$$
$$
F_{0}([b_1])+F_{0}([b_2])=[F(b_1)]+[F(b_2)]=[F(b_{1})+F(b_{2})],
$$
there is an isomorphism $F_{+}(b_1,b_2):F(b_1+b_2)\rightarrow
F(b_1)+F(b_2)$, such that $[F(b_1+b_2)]=[F(b_1)+F(b_2)]$ in
$\pi_{0}(\cC)$. Then $F_0$ is group homomorphism.

Then, for any $[c]\in \pi_{0}(\cC)$, choose a representative element
$c\in obj(\cC)$ of $[c]$. For $c\in obj(\cC)$, and essentially
surjective morphism $F$, there exist $b\in obj{\cB}$ and an
isomorphism $g:F(b)\rightarrow c$ in $\cC$. Then, for $[c]\in
\pi_{0}(\cC)$, there exists $[b]\in \pi_{0}(\cB)$, such that
$F_{0}([b])=[F(b)]=[c]$, i.e. $F_0$ is surjective.
\end{proof}

\begin{Lem}
Given a projective object $P$ in (Ab), where (Ab) is the category of
abelian groups(\cite{10,12}). Then $P_{dis}$ is a projective object
in (2-SGp).
\end{Lem}
\begin{proof}
For each essentially surjective morphism $F:\cA\rightarrow\cB$ and
morphism $G:P_{dis}\rightarrow\cB$ in (2-SGp), from Lemma 2, there
are group homomorphisms $F_{0}:\pi_{0}(\cA)\rightarrow\pi_{0}(\cB)$
and $G_{0}:\pi_{0}(P_{dis})\rightarrow \pi_{0}(\cB)$, with $F_0$ is
a surjection  and $\pi_{0}(P_{dis})=P$.

For $P$ is projective object in (Ab), There exists
$G_{0}^{'}:P\rightarrow\pi_{0}(\cA)$, such that the following
diagram commutes:
\begin{center}
\scalebox{0.9}[0.85]{\includegraphics{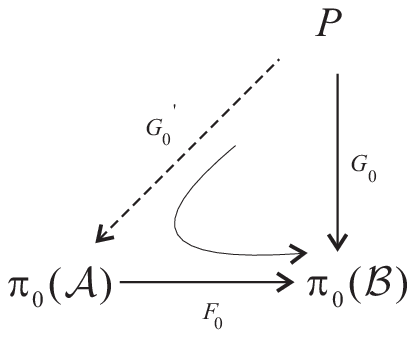}}
\end{center}
From group homomorphism $G_{0}^{'}:P\rightarrow\pi_{0}(\cA)$, define
a morphism
\begin{align*}
&G^{'}:P_{dis}\longrightarrow\cA\\
&\hspace{1.1cm}x\mapsto G^{'}\triangleq G_{0}^{'}(x),\\
&\hspace{0.2cm}x\xrightarrow[]{id}x\mapsto
G_{0}^{'}(x)\xrightarrow[]{id} G_{0}^{'}(x)
\end{align*}
where $G_{0}^{'}(x)$ is the representative element of the equivalent
class $G_{0}^{'}(x)$ in $\pi_{0}(\cA)$, and
$G^{'}(x_{1}+x_{2})=G_{0}^{'}(x_1+x_2)=G_{0}^{'}(x_1)+G_{0}^{'}(x_2)=G^{'}(x_1)+G^{'}(x_2)$,
for $x_1,x_2 \in obj(P_{dis})$.

Moreover, for $x\in obj(P_{dis})=P$, there is
$F_{0}(G_{0}^{'}(x))=G_{0}(x)$, and under the definitions of $F_0$
and $G_0$, we have $[F(G^{'}(x))]=[(F\circ G^{'})(x)]=[G(x)]$ in
$\pi_{0}(\cB)$, then there is an isomorphism $h_{x}:F\circ
G^{'})(x)\rightarrow G(x)$ in $\cB$. It is easy to check that there
is a 2-morphism $h:F\circ G^{'}\Rightarrow G$ in (2-SGp) by
$h_{x}$.

From above, we proved that $P_{dis}$ is a projective object in
(2-SGp).
\end{proof}

\begin{Lem}
Given a projective object $\cP$ in (2-SGp). Then $\pi_{0}(\cP)$ is a
projective object in (Ab).
\end{Lem}
\begin{proof}
For each surjective morphism $f:A\rightarrow B$ and 1-morphism
$g:\pi_{0}(\cP)\rightarrow B$ in (Ab). From Lemma 1, we have an
essentially surjective morphism $F:A_{dis}\rightarrow B_{dis}$ and a
1-morphism $\tilde{G}:(\pi_{0}(\cP))_{dis}\rightarrow B_{dis}$, and
there is a composition $G:\cP\rightarrow
(\pi_{0}(\cP))_{dis}\rightarrow B_{dis}$. There exist a 1-morphism
$G^{'}:\cP\rightarrow A_{dis}$ and a 2-morphism $h:F\circ
G^{'}\Rightarrow G$ in the sense of $\cP$ is projective object in
(2-SGp).

Define a group homomorphism
\begin{align*}
&g^{'}: \pi_{0}(\cP)\longrightarrow A\\
&\hspace{1.2cm}[x]\mapsto g^{'}([x])\triangleq G^{'}(x)
\end{align*}
which is well-defined, since if $[x]=[x^{'}]$ in $\pi_{0}(\cP)$,
there is an isomorphism $\alpha:x\rightarrow x^{'}$ in $\cP$, and
$G^{'}$ is a fuctor, there is a morphism
$G^{'}(\alpha):G^{'}(x)\rightarrow G^{'}(x^{'})$ in $A_{dis}$, so
$G^{'}(x)$ must be equal to $G^{'}(x^{'})$ in $A$, i.e.
$g^{'}([x])=g^{'}([x^{'}])$.

Moreover, from 2-morphism $h:F\circ G^{'}\Rightarrow G$, we have a
morphism $h_{x}:F(G^{'}(x))\rightarrow G(x)$ in $\cB$. Thus, we have
$g^{'}\circ f=g$.
\end{proof}
The next lemma appeared in \cite{6} as a fact without proof, here we
will give its proof.
\begin{Lem}
For a symmetric 2-group $\cA$, there is a full and essentially
surjective 1-morphism $H:\cA\rightarrow (\pi_{0}(\cA))_{dis}$ in
(2-SGp).
\end{Lem}
\begin{proof}

There is a homomorphism of symmetric 2-groups
\begin{align*}
&\hspace{0.3cm}H:\cA\longrightarrow (\pi_{0}(\cA))_{dis}\\
&\hspace{1.1cm}a\mapsto [a],\\
&a_{1}\xrightarrow[]{\alpha}a_{2}\mapsto
[a_1]\xrightarrow[]{id}[a_{2}]
\end{align*}
obviously, $H$ is well-defined homomorphism of symmetric 2-groups.

$H$ is full: for any pair of objects $a_1,a_2$ in $\cA$, and
identity morphism $id:H(a_1)\rightarrow H(a_2)$ in
$(\pi_{0}(\cA))_{dis}$, i.e. $[a_1]=[a_2]$ in $\pi_{0}(\cA)$, and
from the definition of $\pi_{0}(\cA)$, there is an isomorphism
$\alpha:a_1\rightarrow a_2$ in $\cA$, such that $H(\alpha)=id$.

$H$ is essentially surjective: for any object $[a]$ in
$(\pi_{0}(\cA))_{dis}$, choose one representative object $a\in\cA$
of $[a]$, s.t. $H(a)=[a]$.
\end{proof}

Abelian category (Ab) has enough projective objects as the category
of $\mathds{Z}$-modules, i.e. for any abelian group $A$,
there is a surjective morphism $f:P\rightarrow A$, with $P$
projective\cite{10}.

\begin{Thm}
(2-SGp) has enough projective objects, i.e. for any symmetric
2-group in (2-SGp), there is an essentially surjective homomorphism
$F:\cP\rightarrow\cA$, with $\cP$ projective object in (2-SGp).
\end{Thm}
\begin{proof}
For any symmetric 2-group $\cA$, we have an abelian group
$\pi_{0}(\cA)$. Thus, for $\pi_{0}(\cA)\in obj(Ab)$, there is a
surjective morphism $h:P\rightarrow \pi_{0}(\cA)$, with $P$
projective in (Ab). From Lemma 3, we know that $P_{dis}$ is a
projective object in (2-SGp), together with the full and essentially
surjective morphism $H:\cA\rightarrow(\pi_{0}(\cA))_{dis}$\cite{6},
and the 1-morphism $G:P_{dis}\rightarrow (\pi_{0}(\cA))_{dis}$ from
Lemma 1, there exist a 1-morphism $F:P_{dis}\rightarrow \cA$, and
2-morphism $h:H\circ F\Rightarrow G$ as in the following diagram
\begin{center}
\scalebox{0.9}[0.85]{\includegraphics{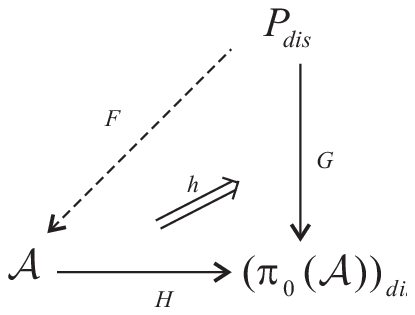}}
\end{center}
Next, we will show that $F:P_{dis}\rightarrow\cA$ is an essentially
surjective morphism  in (2-SGp).

In fact, for any $a\in obj(\cA)$, there is $H(a)\in
obj((\pi_{0}(\cA))_{dis})$, and since $G$ is an essentially
surjective morphism, there exist $x\in obj(P_{dis})$, and
isomorphism $\beta:G(x)\rightarrow H(a)$ in $(\pi_{0}(\cA))_{dis}$.
Using 2-morphism $h:H\circ F\Rightarrow G:P_{dis}\rightarrow
(\pi_{0}(\cA))_{dis}$, there is a morphism $h_{x}:H(F(x))\rightarrow
G(x)$, then we get a composition morphism $\beta\circ
h_{x}:H(F(x))\rightarrow H(a)$ in $(\pi_{0}(\cA))_{dis}$, and since
$H$ is full, there is a morphism $\alpha:F(x)\rightarrow a$ in
$\cA$, such that $H(\alpha)=\beta\circ h_{x}$.

Then for any $a\in obj(\cA)$, there exist $x\in obj({P_{dis}})$ and
an isomorphism $\alpha:F(x)\rightarrow a$ in $\cA$.

Denote $\cP\triangleq P_{dis}$, we have an essentially surjective
morphism $F:\cP\rightarrow\cA$, with $\cP$ projective object in
(2-SGp).
\end{proof}
\section{Main Results II}
\begin{Lem}
For a given 2-ring $\cR$(\cite{4}), $\pi_{0}(\cR)$ is a ring.
\end{Lem}
\begin{proof}From the symmetric 2-group $\cR$, we have an abelian
group $\pi_{0}(\cR)$ which is given as in Lemma 2, together with a
multiplication given by the multiplication of $\cR$, i.e. for
$[r_1],[r_2]$ in $\pi_{0}(\cR)$, $[r_1]\cdot[r_2]\triangleq
[r_1\cdot r_2]$ under the multiplicity of $\cR$. Also, the
multiplicity of $\pi_{0}(\cR)$ satisfies the following conditions,
for all possible elements of $\pi_{0}(\cR)$:

1. $([r_1]\cdot[r_2])\cdot[r_3]=[r_1\cdot r_2]\cdot[r_3]=[(r_1\cdot
rr_2)\cdot r_3]=[r_1(r_2\cdot r_3)]=[r_1]\cdot([r_2]\cdot[r_3])$;

2. There exists $1\in\pi_{0}(\cR)$, which is the unit object in
$\cR$, with $1\cdot[r]=[1\cdot r]=[r]=[r\cdot 1]=[r]\cdot 1$;

3.
$[r]\cdot([s_{0}]+[s_{1}])=[r]\cdot[s_{0}+s_{1}]=[r\cdot(s_{0}+s_{1})]=[r\cdot
s_{0}+r\cdot s_{1}]=[r\cdot s_{0}]+[r\cdot
s_{1}]=[r]\cdot[s_{0}]+[r]\cdot[s_{1}].$

So, $\pi_{0}(\cR)$ is a ring.
\end{proof}

\begin{Lem}
For a ring $R$, there is a 2-ring $R_{dis}$ associated with
$R$.
\end{Lem}
Sketch of proof. $R_{dis}$ is a category consisting of:

$\cdot$ Objects are just the elements of $R$;

$\cdot$ Morphism from $r_{1}$ to $r_2$ is identity if $r_{1}=r_{2}$,
otherwise, empty.

$R_{dis}$ is a discrete symmetric 2-group for $R$ is an abelian
group.

$R_{dis}$ is a 2-ring from $R$ is ring. We can define the 2-ring
structure of $R_{dis}$ the structure of $R$.

\begin{Lem}
Given an $\cR$-2-module $\cM$, then $\pi_{0}(\cM)$ is an
$\pi_{0}(\cR)$-module. Conversely, for an $R$-module $M$, then
$M_{dis}$ is an $R_{dis}$-2-module.
\end{Lem}
\begin{proof}
First, $\pi_{0}(\cM)$ is an abelian group, together with a binary
operator
\begin{align*}
&\cdot:\pi_{0}(\cR)\times\pi_{0}(\cM)\rightarrow \pi_{0}(\cM)\\
&([r],[m])\mapsto [r\cdot m],
\end{align*}
where $r\cdot m$ is the operation of $\cR$ on $\cM$.

Moreover, $(\pi_{0}(\cM),\cdot)$ satisfies:

1.
$[r]\cdot([m_1]+[m_2])=[r]\cdot[m_1+m_2]=[r\cdot(m_1+m_2)]=[r\cdot
m_1+r\cdot m_2]=[r\cdot m_1]+[r\cdot
m_2]=[r]\cdot[m_1]+[r]\cdot[m_2];$

2. $ ([r_1]+[r_2])\cdot[m]=[r_1+r_2]\cdot[m]=[(r_1+r_2)\cdot
m]=[r_1\cdot m+r_2\cdot m]=[r_1\cdot m]+[r_2\cdot
m]=[r_1]\cdot[m]+[r_2]\cdot [m]$;

3. $([r_1]\cdot[r_2])\cdot[m]=[r_1\cdot r_2]\cdot[m]=[(r_1\cdot
r_2)\cdot m]=[r_1\cdot(r_2\cdot m)]=[r_1]\cdot[r_2\cdot
m]=[r_1]\cdot([r_2]\cdot m)$;

4. $1\cdot[m]=[1\cdot m]=[m]$.

So, $\pi_{0}(\cM)$ is an $\pi_{0}(\cR)$-2-module.

Conversely, for an $R$-module $M$, there is a symmetric 2-group
$M_{dis}$. Moreover, there is a bifunctor $\cdot:R_{dis}\times
M_{dis}\rightarrow M_{dis}$ gave by $(r,m)\mapsto r\cdot m$ under
the operation of $R$ on $M$ and natural identities from the axioms
of $R$-module $M$. After basic calculations, $M_{dis}$ is an
$R_{dis}$-2-module.

\end{proof}
\begin{Lem}
Let $f:M\rightarrow N$ be a surjective $R$-homomorphism of
$R$-modules. Then there is an essentially surjective
$R_{dis}$-homomorphism $F:M_{dis}\rightarrow N_{dis}$.
\end{Lem}
\begin{proof}
There ia a functor
\begin{align*}
&F:M_{dis}\rightarrow N_{dis}\\
&m\mapsto F(m)\triangleq f(m),\\
&m\xrightarrow[]{id}m\mapsto F(m)\xrightarrow[]{id}F(m)
\end{align*}
and $F(r\cdot m)\triangleq f(r\cdot m)=r\cdot f(m)=r\cdot F(m)$,
then $F$ is an $R_{dis}$-homomorphism.

For any $n\in N_{dis}=N$, since $f$ is surjective, there exists
$m\in M=obj(M_{dis})$, such that $f(m)=n$, i.e. $F(m)=n$. Then $F$
is an essentially surjective $R_{dis}$-homomorphism.
\end{proof}

\begin{Lem}
Let $F:\cM\rightarrow\cN$ be an essentially surjective
$\cR$-homomorphism of $\cR$-2-modules. Then there is a surjective
$\pi_{0}(\cR)$-homomorphism $f:\pi_{0}(\cM)\rightarrow\pi_{0}(\cN)$.
\end{Lem}
\begin{proof}
There ia a $\pi_{0}(\cR)$-homomorphism
\begin{align*}
&f:\pi_{0}(\cM)\rightarrow\pi_{0}(\cN)\\
&\hspace{1cm}[m]\mapsto f([m])\triangleq [F(m)]
\end{align*}
which is well-defined, if $[m]=[m^{'}]$, i.e. there is an
isomorphism $\alpha:m\rightarrow m^{'}$, then there is an
isomorphism $F(\alpha):F(m)\rightarrow F(m^{'})$ in $\cN$, i.e.
$f(m)=[F(m)]=[F(m^{'})]=f(m^{'})$.

For any $[n]\in \pi_{0}(\cN)$, choose a representative element $n\in
obj(\cN)$ of $[n]$, and since $F$ is essentially surjective, there
exist $m\in obj(\cM)$, and $\beta:F(m)\rightarrow n$. Then there
exists $[m]\in \pi_{0}(\cM)$, such that $f([m])=[F(m)]=[n]$ in
$\pi_{0}(\cN)$, i.e. $f$ is surjective.
\end{proof}
\begin{Lem}
For a projective object $P$ in $(R$-Mod), there is a projective
object $P_{dis}$ in ($R_{dis}$-2-Mod).
\end{Lem}
\begin{proof}
For any essentially surjective $R_{dis}$-homomorphism
$F:\cM\rightarrow\cN$ and $R_{dis}$-homomorphism
$G:P_{dis}\rightarrow \cN$. We have a surjective $R$-homomorphism
$f:\pi_{0}(\cM)\rightarrow\pi_{0}(\cN)$ and an $R$-homomorphism
$g:P\rightarrow \pi_{0}(\cN)$, and $\pi_{0}(P_{dis})=P$ (\cite{6}).

Since $P$ is a projective object, there exists
$g^{'}:P\rightarrow\pi_{0}(\cM)$ such that $f\circ g^{'}=g$. Then we
get an $R_{dis}$-homomorphism
\begin{align*}
&G^{'}:P_{dis}\rightarrow \cM\\
&\hspace{1.2cm}x\mapsto G^{'}\triangleq g^{'}(x)\\
\end{align*}
where $g^{'}(x)$ is the representative element of the isomorphism
class of $g^{'}(x)$ in $\pi_{0}(\cM)$. And from $f(g^{'}(x))=g(x)$,
i.e. $[F(G^{'}(x))]=[G(x)]$, there exists an isomorphism
$h_{x}:F(G^{'}(x))\rightarrow G(x)$ in $\cN$, so defines a
2-morphism $h:F\circ G^{'}\Rightarrow G$.
\end{proof}
\begin{Thm}
($\cR$-2-Mod) has enough projective objects, i.e. for any $\cM\in
obj(\cR$-2-Mod), there exists an essentially surjective
$\cR$-homomorphism $F:\cP\rightarrow\cM$ with $\cP$ projective
object in ($\cR$-2-Mod).
\end{Thm}
\begin{proof}
For $\cM$, $\pi_{0}(\cM)\in obj((\pi_{0}\cR)$-Mod), and
$((\pi_{0}\cR)$-Mod) has enough projective objects(\cite{10}), there
exists a surjective morphism $g:P\rightarrow \pi_{0}(\cM)$ with $P$
projective object in $(\pi_{0}\cR$-Mod). From Lemma 9, we have an
essentially surjective $G:P_{dis}\rightarrow (\pi_{0}(\cM))_{dis}$,
together with full and essentially surjective morphism
$H:\cM\rightarrow (\pi_{0}(\cM))_{dis}$(similar as Lemma 5), and
$P_{dis}$ is a projective object, there exist $F:P_{dis}\rightarrow
\cM$, and 2-morphism $h:H\circ F\Rightarrow G$.

Next, we will check that $F$ is  essentially surjective. For any
$m\in obj(\cM)$, $H(m)\in obj(\pi_{0}(\cM)_{dis})$, and $G$ is
essentially surjective, there exist $x\in obj(P_{dis})=P$ and
isomorphism $\alpha:G(x)\rightarrow H(m)$. From $h:H\circ
F\Rightarrow G$, $h_{x}:H(F(x))\rightarrow G(x)$, together with
$\alpha$, we have the composition morphism $H(F(x))\rightarrow
H(m)$. Moreover, $H$ is full, there exists a morphism
$F(x)\rightarrow m$ in $\cM$.
\end{proof}

\noindent Fang Huang, Shao-Han Chen, Wei Chen\\
Department of Mathematics\\
 South China University of
 Technology\\
 Guangzhou 510641, P. R. China

\noindent Zhu-Jun Zheng\\
Department of Mathematics\\
 South China University of
 Technology\\
 Guangzhou 510641, P. R. China \\
 and\\
Institute of Mathematics\\
Henan University\\  Kaifeng 475001, P. R.
China\\
E-mail: zhengzj@scut.edu.cn


\newpage
\begin{thebibliography}{7}

\bibitem{1} D. Bourn, E.M. Vitale, Extensions of symmetricc at-groups, Homol. Homotopy Appl. 4 (2002) 103¨C162

\bibitem{2} M. Dupont. Abelian categories in dimension 2.PhD.Thesis. arxiv:hep-th/0809.1760v1.

\bibitem{11} A.del R$\acute{I}$o, J. Mart$\acute{I}$nez-Moreno, and E. M. Vitale,
Chain complexes of symmetric categorical groups, J. Pure Appl.
Algebra, 196 (2005).

\bibitem{3} I.P.Freyd. Abelian categories[M]. New York: Harper$\&$Row,
1964.

\bibitem{4} F.Huang, S.H,Chen, W.Chen, Z.J.Zheng. 2-Modules and the Representation of
2-Rings. arxiv:hep-th/1005.2831 18 May 2010

\bibitem{5} M. Jibladze and T. Pirashvili. Third Mac Lane cohomology via
categorical rings [J]. J. Homotopy Relat. Struct., 2 (2007),
pp.187¨C216.

\bibitem{6} S. Kasangian, E.M. Vitale, Factorization systems for symmetric cat-groups, Theory Appl. Categ. 7 (2000)
47¨C70

\bibitem{7} M. Kashiwara, P. Schapira, Categories and Sheaves,Grundlehren der
mathematischen Wissenschaften vol.332, Springer-Verlag.

\bibitem{8} N. T. Quang, D. D. Hanh and N. T. Thuy, On the Axiomatics of Ann-categories [J],
JP J. Algebra Number Theory Appl. 11 (2008), No 1, 59-72.

\bibitem{9} T. Quang, N. T.Thuy and C. T. Kim Phung. The relation between
Ann-categories and ring categories. arXiv:math. CT/0904.1099v1 7 Apr
2009.

\bibitem{10} C. Weibel, An intruction to homological algebra, China
Machine Press,2004.

\bibitem{12} J.J.Rotman, Advanced Modern Algebra,Higer
Education,2004. Press.

\end{thebibliography}
\end{document}